\documentclass[10pt]{amsart}
\usepackage{amssymb,amsmath,mathrsfs,cite}
\usepackage[colorlinks=true,urlcolor=blue,
citecolor=red,linkcolor=blue,linktocpage,pdfpagelabels,bookmarksnumbered,bookmarksopen]{hyperref}
\usepackage[hyperpageref]{backref}
\usepackage[english]{babel}

\usepackage[left=2.65cm,right=2.65cm,top=2.65cm,bottom=2.65cm]{geometry}

\newcommand{\R}{{\mathbb R}}
\newcommand{\eps}{{\varepsilon}}

\newcommand{\norm}[1]{\mathopen\Vert#1\mathclose\Vert}
\DeclareMathOperator{\dist}{dist}

\numberwithin{equation}{section}
\newtheorem{theorem}{Theorem}[section]

\newtheorem{lemma}[theorem]{Lemma}
\newtheorem{remark}[theorem]{Remark}

\newtheorem{definition}[theorem]{Definition}
\theoremstyle{definition}

\title{Symmetry results for the $p(x)$-Laplacian equation}

\author[L.\ Montoro]{Luigi Montoro}
\author[B.\ Sciunzi]{Berardino Sciunzi}
\address{Dipartimento di Matematica,
Universit\`a della Calabria
\newline\indent
Ponte Pietro Bucci 31B, I-87036 Arcavacata di Rende, Cosenza, Italy}
\email{montoro@mat.unical.it}
\email{sciunzi@mat.unical.it}

\author[M.\ Squassina]{Marco Squassina}
\address{Dipartimento di Informatica,
Universit\`a degli Studi di Verona
\newline\indent
C\'a Vignal 2, Strada Le Grazie 15, I-37134 Verona, Italy}
\email{marco.squassina@univr.it}

\thanks{The authors were supported by the 2009
PRIN: {\em Metodi Variazionali e Topologici
nello Studio di Fenomeni non Lineari}}

\begin{document}



\begin{abstract}
We consider the equation $-{\rm div}(|Du|^{p(x)-2}Du)=f(x,u)$ and the related Dirichlet problem.
For axially symmetric domains we prove that, under suitable assumptions, there exist
Mountain-pass solutions which exhibit partial symmetry.
Furthermore, we show that Semi-stable or non-degenerate smooth solutions
need to be radially symmetric in the ball.
\end{abstract}
\maketitle

\section{Introduction and results}

\noindent
Let $\Omega$ be a smooth bounded domain in $\R^N$
and $p:\overline{\Omega}\to \R$ be a continuous function with
\begin{equation}
	\label{prange}
1<p_-:=\inf_{\Omega} p\leq \sup_{\Omega} p=:p_+<\infty.
\end{equation}
In the last few years, the interest towards nonlinear elliptic problems of the type
\begin{equation}
	\label{genfra}
-{\rm div}(|Du|^{p(x)-2}Du)=f(x,u),\quad  \text{in $\Omega$,}
\end{equation}
has considerably increased and various results appeared in the literature about
existence and regularity of weak solutions, see e.g.\ \cite[Chapter 13]{book}
and the references therein. The main goal of our paper is to establish some
symmetry results for positive solutions, provided that the domain $\Omega$ and both
functions $p(x)$ and $x\mapsto f(x,s)$ admit some partial or full symmetry in $\Omega$. We shall obtain
two type of symmetry results by exploiting two completely different techniques. A first class
of results is obtained through suitable versions of the Mountain-pass Theorem which incorporates
symmetry features provided that the functional naturally associated with the problem does increase
under polarization \cite{jvs,jvs-topol,marco-jean,myself}.
In this case we obtain the existence of nontrivial {\em Mountain-pass
solutions} with some partial symmetry information if the domain is axially symmetric
with respect to a fixed half space $H$ with $0\in\partial H$ or if it is invariant under
reflection with respect to any half space $H$ with $0\in\partial H$.
A second class of results is obtained when $\Omega$ is a ball in $\R^N$ by exploiting fine regularity estimates
for the $C^{1,\alpha}$ solutions, allowing to obtain a meaningful definition for
the first eigenvalue of the linearized operator associated with \eqref{genfra},
see \cite{CES1,DS,FMS2,MSS}.
In this case we obtain that any {\em Semi-stable solution}, namely the first eigenvalue
of a suitably defined linearized operator is nonnegative, is radially symmetric when $f(x,s)=f_0(|x|,s)$
and $p(x)=p_0(|x|)$. Whence, in some sense, solutions
with some minimality property such as being of Mountain-pass type or Semi-stable inherit some
symmetry from the data of the problem. We now come to the statement
of the main results. In the following
we denote by $H\subset\R^N$ a closed affine half-space
of $\R^N$, by $\sigma_H(x)$ the reflected of a
point $x\in\R^N$ with respect to $\partial H$ and by ${\mathcal H}_0$ the set of all half spaces
$H\subset \R^N$ such that $0\in\partial H$. The polarization of
$u$ by a half-space $H$ is denoted by $u^H$ and $\sigma_H(\Omega)$ denotes the set of
all reflected points of $\Omega$.

\begin{theorem}
	\label{mainth1}
Assume that $\sigma_H(\Omega)=\Omega$ for some $H\in {\mathcal H}_0$ and, for all $x\in\Omega$
\begin{equation}
	\label{symmc}
p(\sigma_H(x))=p(x), \,\,\quad q(\sigma_H(x))=q(x),\,\,\quad
V(\sigma_H(x))=V(x), \,\,\quad K(\sigma_H(x))=K(x).
\end{equation}
Also, assume that $p,q$ are logarithmic
H\"older continuous and  $q:\overline{\Omega}\to\R$ is a continuous function with
\begin{equation}
	\label{subcritt}
\inf_{x\in\Omega}(q(x)-p(x)+1)>0
\quad
{\rm and}
\quad
\inf_{x\in\Omega}(p^*(x)-q(x)-1)>0,
\quad
p^*(x)=\frac{p(x)N}{N-p(x)},
\end{equation}
$V,K\in C(\overline{\Omega})$ with $V(x)\geq V_0>0$ for all $x\in\Omega$.
Then there exists a nontrivial solution $u\in W^{1,p(x)}_0(\Omega)$ of
\begin{equation}
	\label{problem}
\begin{cases}
-{\rm div}(|Du|^{p(x)-2}Du)+V(x)u^{p(x)-1}=K(x)u^{q(x)} & \text{for $x\in\Omega$,}  \\
u\geq 0 & \text{for $x\in\Omega$,}  \\
u=0 & \text{for $x\in\partial\Omega$.}
\end{cases}
\end{equation}
at the Mountain-pass level such that $u^H$ is also a solution of \eqref{problem} at the same energy level.
\end{theorem}

\noindent
In \cite[Lemma 2.5]{bwwi}, for the semi-linear case $p(x)=2$ for every $x\in\Omega$, the authors
introduce a new ingredient, namely that if $u,u^H$ are both classical solution of $-\Delta w=f(x,w)$
and $f$ satisfies the invariance
\begin{equation}
	\label{invrefl-f}
f(\sigma_H(x),s)=f(x,s),\qquad\text{for all $x\in\Omega$ and $s\in\R$},
\end{equation}
with respect to some $H\in {\mathcal H}_0$,
then either $u(x)>u(\sigma_H(x))$ for all $x\in{\rm Int}(H\cap \Omega)$ (resp.\ $u(x)<u(\sigma_H(x))$ for all $x\in{\rm Int}(H\cap \Omega)$)
or $u(x)=u(\sigma_H(x))$ for all $x\in \Omega$. On account of Theorem~\ref{mainth1}, it would be
interesting to extend these type of results to more general framework. This is to our knowledge
an interesting open problem. In the framework of Theorem \ref{mainth1}, we also have the following

\begin{theorem}
	\label{mainth2}
Assume that $\sigma_H(\Omega)=\Omega$ for all $H\in {\mathcal H}_0$,
and that \eqref{symmc}-\eqref{subcritt} hold. Then there exists
a nontrivial solution $u\in W^{1,p(x)}_0(\Omega)$ of \eqref{problem} at the Mountain-pass level such
that $u(x)=\psi(|x|,\xi\cdot x)$ for some unit vector $\xi\in\R^N$
and some $\psi:\R^+\times\R\to\R$ with $\psi(r,\cdot)$ nondecreasing for all $r\geq 0$.
\end{theorem}

\noindent
The statement of Theorem~\ref{mainth2} could be easily extended,
via minor modifications, to cover the case where the domain is invariant
under spherical cap symmetrization \cite{jvs}, $\Omega^*=\Omega$, which is
equivalent to $\Omega^H=\Omega$ for every $H\in {\mathcal H}_0$,
in place of the more stringent assumption $\sigma_H(\Omega)=\Omega$, for all $H\in {\mathcal H}_0$.
It is readily seen that Theorems~\ref{mainth1} and~\ref{mainth2} can be extended to cover a more general
class of nonlinearities $f(x,s)$ in place of $K(x)s^{q(x)}$ for $s\geq 0$. It is sufficient to assume \eqref{invrefl-f}
a growth condition such as $|f(x,s)|\leq C+C|s|^{q(x)}$ for all $x\in\Omega$ and $s\in\R$, $f(x,s)=0$ for $s\leq 0$
(in order to guarantee that the solutions are nonnegative), $f(x,s)=o(|s|^{p(x)-1})$ as $s\to 0$ and an Ambrosetti-Rabinowitz
type condition: there exists $\mu>0$ with $\inf\{\mu-p(x):x\in\Omega\}>0$ and $R>0$ such that
$\mu F(x,s)\leq f(x,s)s$ for all $x\in\Omega$ and $s\geq R$, where $F(x,s)=\int_0^s f(x,\tau)d\tau$.
We refer the reader to \cite{chabr}, where the
Mountain-pass geometry and the Palais-Smale condition of
\begin{equation*}
\varphi(u)=\int_\Omega\frac{|Du|^{p(x)}}{p(x)}+\int_\Omega \frac{V(x)}{p(x)} |u|^{p(x)}-\int_\Omega F(x,u),\quad
u\in W^{1,p(x)}_0(\Omega),
\end{equation*}
are handled in this framework.
In the second part of the paper we study the radial symmetry of solutions to \eqref{genfra}, considering the problem
\begin{equation}\label{P}
\begin{cases}
-{\rm div}(|Du|^{p(|x|)-2}Du)=f(|x|,u) & \text{in $\Omega$}, \\
u>  0 & \text{in $\Omega$}, \\
u=0 & \text{on $\partial \Omega$}\,.
\end{cases}
\end{equation}
with $f(t,s)$  locally
Lipschitz continuous in $[0,\infty)\times [0,\infty)$ and positive in $[0,\infty)\times (0,\infty)$. Let us recall that
 the corresponding linearized operator is given by
\begin{align}\nonumber
\hspace{1cm} L_u(v,\varphi)&:= \int_\Omega |Du|^{p(x)-2}(Dv, D\varphi)\\
&+\int_\Omega (p(x)-2) |Du|^{p(x)-4}(Du,Dv)(Du,D\varphi) - \int_\Omega \partial_sf(|x|,u)v \varphi, \nonumber
\end{align}
for any $v,\varphi\in H^{1,2}_{0,\rho}$, where the weighted Sobolev space  $H^{1,2}_{0,\rho}$ will be
suitably defined in Section \ref{section4}.
We will prove some summability properties of
$|Du|^{-1}$ that will allow us to get a weighted Sobolev type inequality (see Theorem \ref{thm: Sobolev}).
This is the key to recover a complete  spectral theory for the linearized operator, carried out in Section \ref{hfjhdbvjdhbvjhdf}. Consequently we can give the following

\begin{definition}\rm
	\label{jhfgjdbjf}
We say that a solution $u$ is Semi-stable if
\[
\mu_1(L_u,\Omega)\geq 0
\]
being $\mu_1(L_u,\Omega)$ the first eigenvalue of the linearized operator $L_u$ in $\Omega$.
Furthermore, the solution $u$ is said to be non-degenerate if $0$ is not an eigenvalue of
the linearized operator $L_u$ in $\Omega$.
\end{definition}

Note that, by the variational characterization of the first eigenvalue, it follows that
equivalently $u$ is Semi-stable if and only if
$L_u(\varphi,\varphi)\geq 0$ for any $\varphi\in H^{1,2}_{0,\rho}$. Since the
linearized operator arises as second derivative of the energy
functional, it follows that the minima of the energy functional are Semi-stable solutions.
Also, if  $f(t,s)$ is decreasing with respect to the $s$-variable,
then it follows that any solution is Semi-stable. Moreover in many
cases, depending on $p(\cdot)$, it is possible to show that monotone
solutions are stable (namely $\mu_1(L_u,\Omega)>0$) solutions, see e.g.\ \cite{dino-sns}.
\noindent
On the other hand Mountain-pass solutions (as the ones
previously obtained) generally have Morse index
equal to one. That is, the first eigenvalue of the linearized operator
is negative, and the second one is non-negative. This is well known in the
semi-linear case and we refer to \cite{CES2} for some remarks regarding the quasi-linear case. We have the following
\begin{theorem}\label{radsymm}
Let $\Omega$ be a ball or an annulus  in $\R^N$ and $u$ be any $C^{1,\alpha}(\overline\Omega)$ solution to \eqref{P},
with $f(t,s)$ is locally
Lipschitz continuous in $[0,\infty)\times [0,\infty)$ and
positive in $[0,\infty)\times (0,\infty)$. Assume that the $u$ is Semi-stable.
Then, if $p\in C^1(\Omega)$ with $p(|x|)\geq 2$  then $u$ is radially symmetric.
The same conclusion follows assuming that the solution $u$ in non-degenerate.
\end{theorem}

The symmetry result obtained in  Theorem  \ref{radsymm} holds under very general
assumptions on the nonlinearity $f$, assuming that the solution is
Semi-stable or non-degenerate. In the semi-linear case $p(x)=2$, or more
generally in the quasi-linear case $p(x)=p$, in the case of a  convex
domain (not the annulus), it is possible to get similar results exploiting
the moving plane technique \cite{serrin} (see also \cite{GNN}), without any
stability assumption. We refer to \cite{DS} and the references
therein for a description of the moving planes procedure in the quasi-linear case.
\noindent
Let us mention here that this technique in general can not be exploited in our
case. In fact the moving plane technique is based on the invariance
of the equation under reflections with respect to hyperplanes, which
is not true in general in the case
of $p(x)$-Laplace equations.
\noindent
Let us also point out that our result holds in the case
of solutions which are minima of the associated energy functional (and consequently Semi-stable).
We refer to \cite{fan} (see Section 3) for previous results in this setting.

\section{Recalls on variable exponent Sobolev spaces}

\noindent
We recall here some definitions and basic
properties of the variable exponent Lebesgue-Sobolev spaces
$L^{p(\cdot)}(\Omega)$, $W^{1,p(\cdot)}(\Omega)$ and
$W_0^{1,p(\cdot)}(\Omega)$, where $\Omega$ is a bounded domain in $\R^N$.
We set
$$
C_+(\overline{\Omega}) = \Big\{h\in C(\overline{\Omega}):
\min_{\overline{\Omega}} h> 1 \Big\},
$$
and, for $h\in C(\overline{\Omega})$, we denote
\[
h_- := \min_{\overline{\Omega}} h \quad \text{ and } \quad h_+ := \max_{\overline{\Omega}} h.
\]
For $p \in C_+(\overline{\Omega})$, we introduce the variable exponent Lebesgue space
\[
L^{p(\cdot)}(\Omega) = \left\{u : \Omega \to\R :
\text{ $u$ is measurable and $\int_{\Omega} |u(x)|^{p(x)} dx <+\infty$} \right\},
\]
endowed with the Luxemburg norm
\[
\|u\|_{p(\cdot)} = \inf \Big\{\mu>0:
\int_\Omega \left|\frac {u(x)} \mu\right|^{p(x)} dx \le 1 \Big\},
\]
which is a separable and reflexive Banach space.
If $u \in L^{p(\cdot)}(\Omega)$, the term
$\rho_{p(\cdot)}(u) := \int_{\Omega} |u(x)|^{p(x)} dx$ is called $p(\cdot)$-modular of $u$.
We summarize here a few basic properties of
these spaces, the details being found in \cite{book}.
If $p_1, p_2 \in C_+(\overline{\Omega})$ such that
$p_1 \le p_2$ in $\Omega$, then the embedding
$L^{p_2(\cdot)}(\Omega) \hookrightarrow L^{p_1(\cdot)}(\Omega)$ is continuous.
For any $u \in L^{p(\cdot)}(\Omega)$ and $v \in L^{p'(\cdot)}(\Omega)$,
the following H\"{o}lder type inequality holds
\begin{equation*}
\left|\int_\Omega uv dx\right| \le \left(\frac1{p^-} +
\frac1{p'^-}\right) \|u\|_{p(\cdot)} \|v\|_{p'(\cdot)}.
\end{equation*}
The norm and $p(\cdot)$-modular of every $u \in L^{p(\cdot)}(\Omega)$ have the following relation
\[
\min\big\{\|u\|_{p(\cdot)}^{p^-}, \|u\|_{p(\cdot)}^{p^+}\big\} \le
\rho_{p(\cdot)}(u)
\le \max\big\{\|u\|_{p(\cdot)}^{p^-}, \|u\|_{p(\cdot)}^{p^+}\big\}.
\]
For $p \in C_+(\overline{\Omega})$, the variable exponent Sobolev space is defined by
\[
W^{1,p(\cdot)}(\Omega) = \big\{u \in L^{p(\cdot)}(\Omega) :
D_i u \in L^{p(\cdot)}(\Omega)\,\,\, \text{for $i=1, ..., N$} \big\},
\]
endowed with the norm
\[
\|u\| = \|u\|_{p(\cdot)} + \|D u\|_{p(\cdot)},
\]
which is a separable and reflexive Banach space.
It is important to note that, unlike the constant
exponent case, the smooth functions are in general not dense
in $W^{1,p(\cdot)}(\Omega)$.
However, as shown in \cite{book},
if the exponent variable $p \in C_+(\overline{\Omega})$
is logarithmic Holder continuous, see \cite{book}, then the smooth functions are dense in $W^{1,p(\cdot)}(\Omega)$.
The space $W_0^{1,p(\cdot)}(\Omega)$ is defined as the
closure of $C_0^\infty(\Omega)$ under the norm $\|\cdot\|$, are meaningful.
Moreover, the $p(\cdot)$-Poincar\'e inequality
$\|u\|_{p(\cdot)} \le C \|D u\|_{p(\cdot)}$
holds for all $u \in W_0^{1,p(\cdot)}(\Omega)$,
where $C$ depends on $p$, $|\Omega|$, ${\rm diam}(\Omega)$ and $N$,
see \cite[Theorem 4.3]{book}. Therefore,
\[
\|u\|_{1,p(\cdot)} = \|D u\|_{p(\cdot)}
\]
is an equivalent norm in $W_0^{1,p(\cdot)}(\Omega)$ and
$W_0^{1,p(\cdot)}(\Omega)$ is a separable and reflexive Banach space.
Finally, note that when $s \in C_+(\overline{\Omega})$ and $\inf_{\Omega}(p^*(x)-s(x))>0$, where $p^*(x) = Np(x)/[N - p(x)]$
if $p(x) < N$ and $p^*(x) = \infty$ if $p(x) \ge N$,
the embedding $W_0^{1,p(\cdot)}(\Omega) \hookrightarrow
L^{s(\cdot)}(\Omega)$ is compact.
\vskip6pt
\noindent
{\sc Notation.} Generic fixed numerical constants will be denoted by
$C$ (with subscript in some case), and will be allowed to vary within a
single line or formula.

\section{Proof of Theorem~\ref{mainth1}}
\label{proof12}
Problem~\eqref{problem} is naturally associated with the functional $\varphi:W^{1,p(x)}_0(\Omega)\to\R$
\begin{equation}
	\label{deffunct}
\varphi(u)=\int_\Omega\frac{|Du|^{p(x)}}{p(x)}+\int_\Omega \frac{V(x)}{p(x)} |u|^{p(x)}-\int_\Omega \frac{K(x)}{q(x)+1} |u^+|^{q(x)+1}.
\end{equation}
It is readily seen that $\varphi$ is of class $C^1$ and its critical points correspond to
nonnegative weak solutions to \eqref{problem}, namely we have
$$
\int_{\Omega}|Du|^{p(x)-2}Du\cdot D\zeta+\int_{\Omega}V(x)u^{p(x)-1}\zeta=\int_{\Omega}K(x)u^{q(x)}\zeta,\quad
\forall\zeta\in W^{1,p(x)}_0(\Omega).
$$
\noindent
For the reader's convenience,
we recall that the polarization of a measurable function $u:\R^N\to\R$ by a polarizer $H$
is the function $u^H:\R^N\to\R$ defined by
\begin{equation*}
% \label{polarizationdef}
u^H(x):=
\begin{cases}
 \max\{u(x),u(\sigma_H(x))\}, & \text{if $x\in H$} \\
 \min\{u(x),u(\sigma_H(x))\}, & \text{if $x\in \R^N\setminus H$.} \\
\end{cases}
\end{equation*}
The polarization $\Omega^H\subset\R^N$ of a set $\Omega\subset\R^N$ is
defined as the unique set which satisfies $\chi_{\Omega^H}=(\chi_\Omega)^H$,
where $\chi$ denotes the characteristic function.
The polarization $u^H$ of a function $u$ defined on $\Omega\subset \R^N$
is the restriction to $\Omega^H$ of the polarization of the extension $\tilde u:\R^N\to\R$ of
$u$ by zero outside $\Omega$. For a domain $\Omega$, the set $\sigma_H(\Omega)$ denotes the set of
all reflected points of $\Omega$. In particular, if $H\in {\mathcal H}_0$ and $\Omega$ is invariant under reflection
with respect to $\partial H$, namely $\sigma_H(\Omega)=\Omega$, then $u^H:\Omega\to\R$ writes down as
\begin{equation}
 \label{polari-zz}
u^H(x)=
\begin{cases}
 \max\{u(x),u(\sigma_H(x))\}, & \text{if $x\in H\cap\Omega$} \\
 \min\{u(x),u(\sigma_H(x))\}, & \text{if $x\in (\R^N\setminus H)\cap \Omega$.} \\
\end{cases}
\end{equation}

\subsection{Some preliminary results}

\noindent
In \cite{marco-jean}, Squassina and Van Schaftingen recently proved the following

\begin{lemma}
\label{propositionAbstractSymmetricMinimax}
Let $(X,\|\cdot\|)$ be a Banach space, \(M\) be a metric space and \(M_0 \subset M\).
Let also consider \(\Gamma_0 \subset C(M_0, X)\) and define the set
\[
  \Gamma=\{ \gamma \in C(M, X): \gamma\vert_{M_0} \in \Gamma_0 \}
\]
If \(\varphi \in C^1(X, \R)\)
satisfies
\[
  c=\inf_{\gamma \in \Gamma} \sup_{t\in M} \varphi (\gamma(t)) > \sup_{\gamma_0 \in \Gamma_0} \sup_{t \in M_0} \varphi(\gamma_0(t))=a,
\]
\(\Psi \in C(X, X)\) and
\[
  \varphi \circ \Psi \le \varphi, \qquad \Psi(\Gamma)\subset\Gamma,
\]
then for every \(\epsilon \in ]0, \frac{c-a}{2}[\), \(\delta > 0\) and \(\gamma \in \Gamma\) such that
\[
 \sup_{M} \varphi \circ \gamma \le c+\epsilon,
\]
there exist elements \(u, v, w \in X\) such that
\begin{enumerate}
	%[a)]
 \item[a.1)] \(c-2\epsilon \leq \varphi(u) \leq c+2\epsilon\),
\vskip2pt
 \item[a.2)] \(c-2\epsilon \le \varphi(v) \le c+2\epsilon\),
\vskip2pt
 \item[b.1)] \(\norm{u-w} \le 3\delta\),
\vskip2pt
 \item[b.2)] \(\dist_X(w, \gamma(M)) \le \delta\),
\vskip2pt
 \item[b.3)] \(\norm{v-\Psi(w)} \le 2\delta\),
\vskip2pt
 \item[c.1)] \(\norm{\varphi'(u)} < 8\epsilon/\delta\),
\vskip2pt
 \item[c.1)] \(\norm{\varphi'(v)} < 8\epsilon/\delta\).
\end{enumerate}
\end{lemma}

\noindent
We now prove the following

\begin{lemma}
	\label{invar}
Assume that $\sigma_H(\Omega)=\Omega$ with respect to some $H\in {\mathcal H}_0$
and that $p:\overline{\Omega}\to(1,+\infty)$ and $\mu:\overline{\Omega}\to\R^+$ are continuous functions such that
\begin{equation}
	\label{symmcond}
p(\sigma_H(x))=p(x),\quad \mu(\sigma_H(x))=\mu(x),\quad\text{for all $x\in\Omega$.}
\end{equation}
Then
$$
\int_\Omega \mu(x)|Du^H|^{p(x)}=\int_\Omega \mu(x)|Du|^{p(x)},\qquad \text{for all $u\in W^{1,p(x)}_0(\Omega)$.}
$$
Similarly
$$
\int_\Omega \mu(x)|u^H|^{p(x)}=\int_\Omega \mu(x)|u|^{p(x)},\qquad \text{for all $u\in L^{p(x)}(\Omega)$.}
$$
\end{lemma}
\begin{proof}
	If $u\in W^{1,p(x)}_0(\Omega)$ and $H\in {\mathcal H}_0$, it follows
	that $u^H\in W^{1,p(x)}_0(\Omega)$. To prove this, it is sufficient to argue as in the beginning of the proof of
	\cite[Proposition 2.3]{smetswillem} for the case $\Omega=\R^N$ and then recall that by definition
    $u^H=(\tilde u)^H|_{\Omega}$ and $(\tilde u)^H|_{\R^N\setminus\Omega}=0$, being $\sigma_H(\Omega)=\Omega$.
    Setting $v(x):=u(\sigma_H(x))$ and $w(x):=u^H(\sigma_H(x))$, it follows that $v,w$ belong to $W^{1,p(x)}_0(\Omega)$ and
	\begin{equation}
		\label{derivform}
		D u^H(x)=
		\begin{cases}
			D u(x) & \text{if $x\in \{u>v\}\cap H\cap \Omega$}, \\
			D v(x) & \text{if $x\in \{u\leq v\}\cap H \cap \Omega$},
		\end{cases}
		\qquad
		D w(x)=
		\begin{cases}
			D v(x) & \text{if $x\in \{u>v\}\cap H\cap \Omega$}, \\
			D u(x) & \text{if $x\in \{u\leq v\}\cap H\cap \Omega$}.
		\end{cases}
	\end{equation}
	and, for $x\in H\cap\Omega$, we have
	$u^H(x)=v(x)+(u(x)-v(x))^+$ and $w(x)=u(x)-(u(x)-v(x))^+$.
	Writing down $\sigma_H$ as $\sigma_H(x)=x_0+Rx$, where $R$ is an orthogonal linear transformation (symmetric, as reflection),
	taking into account that $|{\rm det}R|=1$ and
	$|D v(x)|=|D (u(\sigma_H(x)))|=|R(D u(\sigma_H(x)))|=|(D u)(\sigma_H(x))|$
	(and the analogous formula for $|Dw(x)|=|(D u^H)(\sigma_H(x))|$) recalling \eqref{symmcond}, \eqref{derivform} and that
	$$
	\Omega\cap(\R^N\setminus H)=\sigma_H(\Omega\cap H),
	$$
	we have
	\begin{align*}
		\int_\Omega \mu(x)|Du|^{p(x)} & =\int_{H\cap\Omega}\mu(x)|Du|^{p(x)}+\int_{H\cap\Omega}\mu(x)|(D u)(\sigma_H(x))|^{p(x)} \\
&		=\int_{H\cap\Omega}\mu(x)|Du|^{p(x)}+\int_{H\cap\Omega}\mu(x)|D v|^{p(x)} \\
&		=\int_{\{u>v\}\cap H\cap\Omega}\mu(x)|Du|^{p(x)}+\int_{\{u>v\}\cap H\cap\Omega}\mu(x)|Dv|^{p(x)} \\
&		+\int_{\{u\leq v\}\cap H\cap\Omega}\mu(x)|Dv|^{p(x)}+\int_{\{u\leq v\}\cap H\cap\Omega}\mu(x)|Du|^{p(x)} \\
&	=\int_{H\cap\Omega}\mu(x)|Du^H|^{p(x)}+\int_{H\cap\Omega}\mu(x)|Dw|^{p(x)}=\int_{\Omega}\mu(x)|Du^H|^{p(x)}.
	\end{align*}
	This concludes the proof.
\end{proof}

\vskip5pt
\noindent
We can now prove the following

\begin{lemma}
	\label{continpol}
Assume that $\sigma_H(\Omega)=\Omega$ with respect to some $H\in {\mathcal H}_0$	
and that $p:\overline{\Omega}\to(1,+\infty)$ is a continuous functions such that
\begin{equation}
	\label{symmcondbis}
p(\sigma_H(x))=p(x),\qquad\text{for all $x\in\Omega$.}
\end{equation}
Then the map
$$
\Psi:W^{1,p(x)}_0(\Omega)\to W^{1,p(x)}_0(\Omega),\qquad
u\mapsto u^H
$$
is well defined and continuous.
\end{lemma}
\begin{proof}
Let $(u_j)\subset W^{1,p(x)}_0(\Omega)$ be a sequence which strongly converges to some $u_0\in W^{1,p(x)}_0(\Omega)$.
Observe that, for every fixed $\lambda>0$, by applying Lemma~\ref{invar} with $\mu(x):=\lambda^{-p(x)}$ we have
\begin{equation}
	\label{uguaglianz}
\int_\Omega \Big(\frac{|Du^H_j|}{\lambda}\Big)^{p(x)}=
\int_\Omega \Big(\frac{|Du_j|}{\lambda}\Big)^{p(x)},\qquad \text{for all $j\geq 1$.}
\end{equation}
Then, by the arbitrariness of $\lambda$ and the definition on $\|\cdot\|_{{L^{p(x)}}}$, there holds
$$
\sup_{j\geq 1}\|Du_j^H\|_{L^{p(x)}}=\sup_{j\geq 1}\|Du_j\|_{L^{p(x)}}<+\infty.
$$
Since $(u_j^H)$ is bounded in the reflexive space $W^{1,p(x)}_0(\Omega)$, up to a subsequence, there
exists $w\in W^{1,p(x)}_0(\Omega)$ such that $(u_j^H)$ converges weakly to $w$
as $j\to\infty$. Observe now that, since the polarization is contractive for $L^m(\Omega)$-spaces (precisely,
see \cite[Proposition 2.3]{jvs}, case of totally invariant domains) and
since the injection $i:L^{p(x)}(\Omega)\to L^{p_-}(\Omega)$ is continuous, for all $j\geq 1$
\begin{align*}
\|u_j^H-u_0^H\|_{L^{p_-}(\Omega)}\leq \|u_j-u_0\|_{L^{p_-}(\Omega)}\leq C\|u_j-u_0\|_{L^{p(x)}(\Omega)}\leq
C\|u_j-u_0\|_{W^{1,p(x)}_0(\Omega)},
\end{align*}
where in the last inequality we used Poincar\'e inequality.
Hence $u_j^H$ converges to $u_0^H$ strongly in $L^{p_-}(\Omega)$.
Hence $w=u_0^H$. In conclusion
$$
u_j^H\rightharpoonup u_0^H\,\,\,\,\text{in $W^{1,p(x)}_0(\Omega)$ as $j\to\infty$},
\quad
{\rm and}
\quad
\lim_{j\to\infty}\|Du^H_j\|_{L^{p(x)}(\Omega)}=\|Du^H_0\|_{L^{p(x)}(\Omega)}.
$$
Since $W^{1,p(x)}_0(\Omega)$ is uniformly convex (see, for instance, \cite[Theorem 8.1.6,
p.243]{book}), we can finally conclude that $u_j^H\to u_0^H$ as $j\to\infty$ in $W^{1,p(x)}_0(\Omega)$.
\end{proof}

\subsection{Proof of Theorem~\ref{mainth1} concluded}
With the above results, apply Lemma~\ref{propositionAbstractSymmetricMinimax} by taking
\begin{equation}
	\label{frameworkMP}
X:=W^{1,p(x)}_0(\Omega),\quad M:=[0,1],\quad M_0:=\{0,1\},\quad\Gamma_0=\{0,\xi\}
\end{equation}
with $\xi\geq 0$ a fixed function with $\xi^H=\xi$ and $\varphi(\xi)<0$ (for an explicit
construction of a function $\xi$ satisfying these conditions, see \cite[bottom of p.613]{chabr}) and hence
$$
\Gamma=\big\{\gamma\in C([0,1],W^{1,p(x)}_0(\Omega)):\gamma(0)=0,\,\,\gamma(1)=\xi\}.
$$
It is readily seen that the functional $\varphi$ introduced in \eqref{deffunct} is $C^1$ smooth.
Furthermore,
\[
  c=\inf_{\gamma \in \Gamma} \sup_{t\in [0,1]} \varphi (\gamma(t))>0=\max\{\varphi(0),\varphi(\xi)\}=
\sup_{\gamma_0 \in \{0,\xi\}} \sup_{t \in \{0,1\}} \varphi(\gamma_0(t))=a.
\]
where the first inequality (namely the Mountain-pass geometry of $\varphi$)
can be proved by arguing exactly as in \cite[pp.612-613]{chabr}. In light of Lemma~\ref{continpol}
the polarization map is continuous. Also by using again Lemma~\ref{invar} with the choices $\mu(x)=p(x)^{-1}$,
$\mu(x)=\frac{V(x)}{p(x)}$ and $\mu(x)=\frac{V(x)}{q(x)+1}$ respectively (notice that, on account of \eqref{symmc}
any of these choices of $\mu$ remain invariant under reflection with respect to $\partial H$), we have
\begin{align*}
\varphi(u^H)&=\int_\Omega\frac{|Du^H|^{p(x)}}{p(x)}+\int_\Omega \frac{V(x)}{p(x)} |u^H|^{p(x)}-\int_\Omega \frac{K(x)}{q(x)+1} |(u^+)^H|^{q(x)+1} \\
&=\int_\Omega\frac{|Du|^{p(x)}}{p(x)}+\int_\Omega \frac{V(x)}{p(x)} |u|^{p(x)}-\int_\Omega \frac{K(x)}{q(x)+1} |(u^+)|^{q(x)+1}=\varphi(u)
\end{align*}
for every $u\in W^{1,p(x)}_0(\Omega)$. Finally, $\Psi(\Gamma)\subset\Gamma$ since for every $\gamma\in\Gamma$ it follows,
again in view of Lemma~\ref{continpol}, that $\gamma^H\in C([0,1],W^{1,p(x)}_0(\Omega))$ and
$\gamma^H(0)=(\gamma(0))^H=0^H=0$ and $\gamma^H(1)=(\gamma(1))^H=\xi^H=\xi$.
By the definition of $c$ we can find a sequence of curves $(\gamma_j)\subset\Gamma$ such that
\[
 \sup_{t\in [0,1]} \varphi(\gamma_j([0, 1])) \le c + 1/j^2.
\]
Apply now Lemma~\ref{propositionAbstractSymmetricMinimax} with
\(\delta_j=1/j\), \(\eps_j= 1/j^2\) and
and obtain three sequences \((u_j)\), \( (v_j)\)
and \( (w_j)\) in \(W^{1,p(x)}_0(\Omega)\) with
$\lim\limits_{j} \varphi (u_j) = \lim\limits_{j} \varphi (v_j) = c$,
$\lim\limits_{j} \varphi' (u_j) = \lim\limits_{j} \varphi' (v_j) = 0$
and
$$
\lim_{j} \norm{u_j - w_j}_{W^{1,p(x)}_0(\Omega)} =0,\qquad
\lim_{j} \norm{v_j - w_j^H}_{W^{1,p(x)}_0(\Omega)} = 0.
$$
Since \(\varphi\) satisfies the Palais-Smale condition (to this regard, we refer the reader to
\cite[pp.614-615]{chabr}, our functional is included in the framework covered therein), up to a subsequence, \( (u_j)\) converges
to some $u\in W^{1,p(x)}_0(\Omega)$. Hence, the sequence \((w_j)\) also converges to \(u\). By continuity
of the polarization, \((v_j)\)  converges to \(u^H\). The conclusion follows since $\varphi$ is of class $C^1$. \qed

\section{Proof of Theorem~\ref{mainth2}}

\noindent
We recall a definition from~\cite{jvs}.
Let $X$ and $V$ be two Banach spaces and $S\subset X$.
We consider two maps $*:S\to V$, $u\mapsto u^*$
({\em symmetrization map}) and $h:S\times {\mathcal H}_0\to S$,
$(u,H)\mapsto u^H$ ({\em polarization map}), where ${\mathcal H}_0$
is a path-connected topological space. We assume:
\begin{enumerate}
 \item $X$ is continuously embedded in $V$;
 \item $h$ is a continuous mapping;
\item for each $u\in S$ and $H\in {\mathcal H}_0$ it holds $(u^*)^H=(u^H)^*=u^*$ and $u^{HH}=u^H$;
\item there exists a sequence $(H_m)$ in ${\mathcal H}_0$ such that, for $u\in S$, $u^{H_1\cdots H_m}$ converges
to $u^*$ in $V$;
\item for every $u,v\in S$ and $H\in {\mathcal H}_0$ it holds
$\|u^H-v^H\|_V\leq \|u-v\|_V$.
\end{enumerate}
%Furthermore $*:S\to V$ can be extended to the whole space $X$ by
%setting $u^*:=(\Theta(u))^*$ for all $u\in X$, where $\Theta:(X,\|\cdot\|_V)\to (S,\|\cdot\|_V)$ is
%a Lipschitz function such that $\Theta|_{S}={\rm Id}|_{S}$.

We recall the main result of \cite{jvs}.

\begin{lemma}
\label{mpconc-symm}
Let $X$ and $V$ be two Banach spaces, $S\subset X$, $*$ and ${\mathcal H}_0$
satisfying the requirements of the abstract symmetrization framework.
Let $\varphi:X\to\R$ a $C^1$ functional
Let $M$ be a metric space and $M_0$ a closed subset of $M$
and $\Gamma_0\subset C(M_0,X)$. Let us define
$$
\Gamma=\big\{\gamma\in C(M,X):\,\,\,\gamma|_{M_0}\in \Gamma_0\big\}.
$$
Assume that
$$
+\infty>c=\inf_{\gamma\in\Gamma}\sup_{\tau\in M} \varphi(\gamma(\tau))
>\sup_{\gamma_0\in \Gamma_0}\sup_{\tau\in M_0} \varphi(\gamma_0(\tau))=a,
$$
and that
$$
\forall H\in {\mathcal H}_0,\,\, \forall u\in S:\quad
\varphi(u^H)\leq \varphi(u).
$$
Then, for every $\eps\in(0,(c-a)/2)$, every $\delta>0$ and $\gamma\in \Gamma$ such that
$$
\sup_{\tau\in M}\varphi(\gamma(\tau))\leq c+\eps,\quad \gamma(M)\subset S,\quad
\text{$\gamma|_{M_0}^{H_0}\in\Gamma_0$ for some $H_0\in {\mathcal H}_0$},
$$
there exists $u\in X$ such that
\begin{equation*}
 c-2\eps\leq \varphi(u)\leq c+2\eps,\quad
\|d\varphi(u)\|\leq 8\eps/\delta,\quad
\|u-u^*\|_V\leq K\delta,
\end{equation*}
being $K$ a constant depending upon the embedding $i:X\to V$,
\end{lemma}

\begin{lemma}
	\label{compatib}
	Assume that $\sigma_H(\Omega)=\Omega$ for all $H\in {\mathcal H}_0$
	and that~\eqref{symmc} holds for any $H\in {\mathcal H}_0$.
Then the choice $X:=S=W^{1,p(x)}_0(\Omega)$ and $V:=L^{p_-}(\Omega)$ endowed with the natural norms is compatible
with abstract symmetrization framework.	
\end{lemma}
\begin{proof}
Since $\Omega$ is invariant under reflection with respect to all $H\in {\mathcal H}_0$,
it follows $\Omega$ is invariant under cap symmetrization \cite{jvs}.	
Of course $X$ is continuously embedded into $V$. Let us now prove that $h(u,H):=u^H$ is a continuous mapping
from $X\times {\mathcal H}_0$ to $X$. Here ${\mathcal H}_0$ is meant to be endowed with the metric $d$
introduced in \cite[Definition 2.35]{jvs-topol}, which makes ${\mathcal H}_0$ a separable metric space.
Let $(u_j,H_j)$ be a sequence in	$X\times {\mathcal H}_0$ which
converges to $(u_0,H_0)$. As for identity \eqref{uguaglianz}, for every $\lambda>0$
\begin{equation*}
\int_\Omega \Big(\frac{|Du^{H_j}_j|}{\lambda}\Big)^{p(x)}=
\int_\Omega \Big(\frac{|Du_j|}{\lambda}\Big)^{p(x)},\qquad \text{for all $j\geq 1$.}
\end{equation*}
Then, it follows that $(u^{H_j}_j)$ remains bounded in
$X$ and, up to a subsequence, it converges
to some function $w$ weakly in $X$ (and strongly in $V$
by the compact embedding theorem).
In particular, $(u_j^{H_j})$ converges to $w$ in $L^{p_-}(\Omega)$. On the other hand,
if $(\vartheta_m)\subset C^\infty_c(\Omega)$ is a sequence converging
to $u_0$ strongly in $L^{p_-}(\Omega)$ as $m\to\infty$, for every $j,m\geq 1$, we have
\begin{align*}
\|u_j^{H_j}-u_0^{H_0}\|_{L^{p_-}(\Omega)}&\leq \|u_j^{H_j}-u_0^{H_j}\|_{L^{p_-}(\Omega)}+
\|u_0^{H_j}-u_0^{H_0}\|_{L^{p_-}(\Omega)} \\
&\leq \|u_j-u_0\|_{L^{p_-}(\Omega)}+
\|u_0^{H_j}-\vartheta_m^{H_j}\|_{L^{p_-}(\Omega)} \\
&+
\|\vartheta_m^{H_j}-\vartheta_m^{H_0}\|_{L^{p_-}(\Omega)} +
\|\vartheta_m^{H_0}-u_0^{H_0}\|_{L^{p_-}(\Omega)} \\
&\leq C\|u_j-u_0\|_{L^{p(x)}(\Omega)}+
2\|\vartheta_m-u_0\|_{L^{p_-}(\Omega)}
+\|\vartheta_m^{H_j}-\vartheta_m^{H_0}\|_{L^{p_-}(\Omega)}.
\end{align*}
Letting $j\to\infty$ at $m$ fixed first and then finally $m\to\infty$, it follows that
$(u_j^{H_j})$ converges to $u_0^{H_0}$ in $L^{p_-}(\Omega)$. We also used the fact that
for a fixed compactly supported function $\vartheta$, it holds
$\vartheta^{H_j}$ converges to $\vartheta^{H_0}$ uniformly on $\Omega$ for $j\to\infty$.
By uniqueness, $w=u_0^{H_0}$.  In conclusion
$$
u_j^{H_j}\rightharpoonup u_0^{H_0},\,\,\,\,\text{as $j\to\infty$},
\quad
{\rm and}
\quad
\lim_{j\to\infty}\|Du^{H_j}_j\|_{L^{p(x)}(\Omega)}=\|Du^{H_0}_0\|_{L^{p(x)}(\Omega)}.
$$
Then, since as already remarked $W^{1,p(x)}_0(\Omega)$ is uniformly
convex, we can conclude that $u_j^{H_j}\to u_0^{H_0}$ as $j\to\infty$,
concluding the proof of the continuity of $h$.
Also, for all $u\in X$, $u$ belongs to $L^{p_-}(\Omega)$ and,
in light of~\cite[Theorem 2.1]{jvs}, there exists a sequence $(H_j)\subset {\mathcal H}_0$
such that, for all $u\in L^{\rho}(\Omega)$, $\|u^{H_1\cdots H_j}-u^*\|_{L^{p_-}}\to 0$.
The contractivity of $u^H$ is the space $L^{p_-}(\Omega)$ is a standard fact.
\end{proof}

\subsection{Proof of Theorem~\ref{mainth2} concluded}
On account of Lemma~\ref{compatib}, it is sufficient to argue as for the proof of
Theorem~\ref{mainth1}. Applying Lemma~\ref{mpconc-symm} with the choices \eqref{frameworkMP},
and with \(\delta_j=1/j\) and \(\eps_j= 1/j^2\),
we find $(u_j)\subset W^{1,p(x)}_0(\Omega)$ such that
$\varphi (u_j)\to c$ and  $\varphi' (u_j)\to 0$ as $j\to\infty$
and $\|u_j - u_j^*\|_{L^{p_-}(\Omega)}\to 0$ as $j\to\infty$.
Since, as already pointed out in the proof of Theorem~\ref{mainth1}, \(\varphi\) satisfies the Palais-Smale condition,
up to a subsequence, \( (u_j)\) converges to some $u \in W^{1,p(x)}_0(\Omega)$. Hence $\varphi(u)=c$ and $\varphi'(u)=0$.
Finally, since
\begin{align*}
\|u - u^*\|_{L^{p_-}(\Omega)}&\leq \|u - u_j\|_{L^{p_-}(\Omega)}+\|u_j - u_j^*\|_{L^{p_-}(\Omega)}+\|u^*-u_j^*\|_{L^{p_-}(\Omega)} \\
&\leq 2C\|u - u_j\|_{L^{p(x)}(\Omega)}+\|u_j - u_j^*\|_{L^{p_-}(\Omega)},
\end{align*}
taking into account Poincar\'e inequality, letting $j\to\infty$, yields $u=u^*$.
This concludes the proof. \qed

\section{Proof of Theorem~\ref{radsymm}}

We consider $C^{1,\alpha}$ solutions to problem \eqref{P}. Obviously problem \eqref{P} has to be understood in weak sense, that is
$u \in W^{1,p(\cdot)}_0(\Omega)$  is a  weak solution to \eqref{P} if
\begin{equation}
	\label{eq:fgakghfhksajdgfja}
\int_\Omega{|Du|^{p(x)-2}(Du},  D\varphi)  = \int_\Omega f(|x|,u) \varphi, \qquad\forall \varphi\in C^1_c(\Omega).
\end{equation}
Throughout  this section we shall always assume the assumptions of Theorem~\ref{radsymm}.
\subsection{A summability result}  We have the following

\begin{lemma}\label{le:1/gradient}
Let $u\in C^{1,\alpha}(\overline \Omega)$ be a positive solution  to \eqref{P}.  Then
\[
\int_{\Omega}\frac{1}{|D u|^{(p(x)-1)r}|x-y|^{\gamma}} \leq C,
\]
where $C$ is a positive constant independent of $y$,
$0 \le r < 1$, $\gamma < N-2 $ if $N\geq3$ and $\gamma =0$ if $N=2$.
In particular it follows that the critical set $Z_u=\{x \in \Omega : |D u(x)|=0\}$ has zero Lebesgue measure.
\end{lemma}

\begin{proof}
We consider, for $y\in\R^N$, the  test function
$$
\psi_\varepsilon(x)= (\varepsilon + |Du|^{(p(x)-1)r})^{-1}\eta {(\varepsilon +|x-y|)^{-\gamma}},
$$
where $\eta$ is a positive smooth cut-off function
with ${\rm supt}(\eta)=\Omega_0$ such that $\eta=1$
on $\tilde\Omega_0\subset\Omega_0$ and $\tilde\Omega_0\subset \subset \Omega$ is such that
$(\Omega \setminus \tilde\Omega_0)\cap Z_u= \emptyset$.
In fact, we recall that, in light of the Hopf boundary Lemma of \cite{Zha},  we have $Z_u\cap\partial\Omega=\emptyset$.
Note that $\psi_\varepsilon$ is a good test function since it belongs to $W^{1,2}(\Omega)$ by the  summability
properties of the solutions proved in   \cite{CLS} and thus it can be plugged into \eqref{eq:fgakghfhksajdgfja} by density arguments.
\noindent
Again by the Hopf boundary Lemma,
to achieve the conclusion, it is enough to show  that \begin{equation}\label{eq:agshaghagsahg}
\int_{\tilde\Omega_0}\frac{1}{|D u|^{(p(x)-1)r}|x-y|^{\gamma}} \leq C\,,
\end{equation}
 for $\tilde\Omega_0\subset \subset \Omega$.
Moreover, without loss of generality,  we can reduce to consider the case
\begin{equation}\label{eq:dhjwvfqjhvfdjaVJ}
\max_{x \in \bar{\Omega}_0}\frac{p(x)-2}{p(x)-1}\leq r<1.
\end{equation}
In fact, once \eqref{eq:agshaghagsahg} holds for $C^{1,\alpha}$ solutions,
the same estimation easily follows for  $r'<r$.
We put $\psi_\eps$ as test function in \eqref{P} and since
$f(|x|,u) \geq \sigma$ for some $\sigma>0$ in the support of $\psi_\eps$, we get
\begin{align*}
&\sigma\int_{\Omega_0}\frac{\eta}{(\varepsilon + |Du|^{(p(x)-1)r})\,(\varepsilon +|x-y|)^{\gamma}}
\leq \int_{\Omega_0} f(|x|,u)\psi_\varepsilon  \\
&\leq  \int_{\Omega_0} |Du|^{p(x)-2} |(Du,D\psi_\varepsilon)|  \\
&\leq  \int_{\Omega_0} (p(x)-1)r\frac{|Du|^{p(x)-2}}{(\varepsilon+ |Du|^{(p(x)-1)r})^2}
 |Du|^{(p(x)-1)r}\frac{1}{(\varepsilon +|x-y|)^\gamma} \eta \|D^2u\|  \\
 &+ \int_{\Omega_0}r |\log|Du||
\frac{|Du|^{p(x)-2}}{(\varepsilon+ |Du|^{(p(x)-1)r})^2}
  |Du|^{(p(x)-1)r+1}
\frac{1}{(\varepsilon +|x-y|)^\gamma}\eta |Dp|\\
&+\int_{\Omega_0} \frac{|Du|^{p(x)-2}}{
(\varepsilon + |Du|^{(p(x)-1)r})}\frac{|Du|\, |D\eta|}{(\varepsilon +|x-y|)^\gamma} \\
 &+\int_{\Omega_0} \gamma \frac{|Du|^{p(x)-2}}{
(\varepsilon + |Du|^{(p(x)-1)r})} \frac{\eta |D u|}{(\varepsilon +|x-y|)^{(\gamma+1)}}.
\end{align*}
Since the critical set $Z_u$ is the zero level set of  $|D u|^{(p(x)-1)r}$, then by Stampacchia's Theorem
the gradient of $|D u|^{(p(x)-1)r}$ vanishes a.e. $Z_u$. In the above calculations we consequently agree that the term $\log |D u|$ make sense outside $Z_u$, while in $Z_u$ the distributional derivatives of $|D u|^{(p(x)-1)r}$ are zero.
\noindent
Taking into account  that
$|\log t| \le C_\delta + t^\delta + t^{-\delta}, \,\, t>0
$ for all $\delta>0$  and some  $C_\delta>0$,  we have
\begin{align}\label{eq:kadhkjhjka}
& \sigma\int_{\Omega_0} \frac{\eta}{(\varepsilon + |Du|^{(p(x)-1)r})(\varepsilon +|x-y|)^{\gamma}}
\\\nonumber &\leq  \int_{\Omega_0} (p(x)-1)r\frac{|Du|^{p(x)-2}}{(\varepsilon+ |Du|^{(p(x)-1)r})^2}
 |Du|^{(p(x)-1)r}\frac{1}{(\varepsilon +|x-y|)^\gamma} \eta \|D^2u\|
\\\nonumber&+ C\int_{{\Omega_0}}\frac{|D u|^{p(x)-1}}{(\varepsilon + |D u|^{(p(x)-1)r})}\frac{1}{(\varepsilon +|x-y|)^{\gamma}}\\
\nonumber&+   C\int_{{\Omega_0}}\frac{|D u|^{p(x)-1+\delta}}{(\varepsilon + |D u|^{(p(x)-1)r})}  \frac{1}{(\varepsilon +|x-y|)^{\gamma}}\\
\nonumber&+ C\int_{{\Omega_0}} \frac{|D u|^{p(x)-1-\delta}}{(\varepsilon + |D u|^{(p(x)-1)r})}  \frac{1}{(\varepsilon +|x-y|)^{\gamma}}\\\nonumber
&+C\int_{{\Omega_0}}\frac{|D u|^{p(x)-1}}{(\varepsilon + |D u|^{(p(x)-1)r})}  \frac{1}{(\varepsilon +|x-y|)^{\gamma+1}}+C,
\end{align}
where $\delta$ was fixed small depending on the size of $p_-$.
Since $u\in C^{1,\alpha}$ and  $\gamma < N-2$,
from \eqref{eq:kadhkjhjka} we get
\begin{equation*}
 \sigma\int_{\Omega_0} \frac{\eta}{(\varepsilon + |Du|^{(p(x)-1)r})(\varepsilon +|x-y|)^{\gamma}}
\leq  C\int_{\Omega_0}\frac{|Du|^{(p(x)-2)+(p(x)-1)r}}{(\varepsilon+ |Du|^{(p(x)-1)r})^2}
 \frac{\eta \|D^2u\|}{(\varepsilon +|x-y|)^\gamma} +C.
\end{equation*}
If $\beta\in C(\bar{\Omega}_0)$ is such that
$\beta(x)=1-(p(x)-1)(1-r)$, with $0\leq \beta(x) <1$
by virtue of \eqref{eq:dhjwvfqjhvfdjaVJ}, by writing
$$
\frac{|Du|^{(p(x)-2)+(p(x)-1)r}}{(\varepsilon+ |Du|^{(p(x)-1)r})^2}
 \frac{\eta \|D^2u\|}{(\varepsilon +|x-y|)^\gamma}
=
\Big[\frac{|Du|^{\frac{3r(p(x)-1)}{2}}}{(\varepsilon+ |Du|^{(p(x)-1)r})^2}
 \frac{\eta^{1/2}}{(\varepsilon +|x-y|)^{\gamma/2}} \Big]\Big[\frac{\eta^{1/2} |Du|^{\frac{p(x)-2-\beta(x)}{2}}\|D^2u
\|}{(\varepsilon +|x-y|)^{\gamma/2}}\Big]
$$
and using a weighted Young inequality, we finally obtain
\begin{gather*}
	%\label{eq:kadaahjhjhkjhjka}
 \sigma\int_{\Omega_0} \frac{\eta}{(\varepsilon + |Du|^{(p(x)-1)r})(\varepsilon +|x-y|)^{\gamma}} \\
\leq \delta' \int_{\Omega_0} \frac{\eta}{(\varepsilon + |Du|^{(p(x)-1)r})(\varepsilon +|x-y|)^{\gamma}}
+\frac {C}{\delta'} \int_{\Omega_0}  |Du|^{p(x)-2-\beta(x)} \|D^2u\|^2
\frac{1}{(\varepsilon +|x-y|)^{\gamma}}+C.
\end{gather*}
Recalling now that, by a variant argument of \cite[Lemma 3.1]{CLS}, we have
\[
\int_{\Omega_0}\frac{|D u|^{p(x)-2-\beta(x)}\|D^2u\|^2}{|x-y|^{\gamma}} \leq C,
\]
by  choosing  $\delta' < \sigma$ we have the desired conclusion  letting $\eps\to 0^+$
and recalling that $\eta=1$ on $\tilde\Omega_0$.
\end{proof}

\subsection{A weighted Sobolev inequality}\label{section4}

Given a solution $u$ to problem \eqref{P}, for $p(x)\geq 2$ we set
\begin{equation*}
	%\label{eq:khajDKajfjh}
\rho(x) = |Du(x)|^{p(x)-2}, \quad x\in \Omega,
\end{equation*}
and  define the
Hilbert space $H^{1,2}_\rho(\Omega)$
as the completion of $C^\infty(\Omega)$ with respect to the norm
\begin{equation*}
	%\label{normaspazio}
\|v\|^2_{H^{1,2}_\rho}=\int_{\Omega}v^2+\int_{\Omega}\rho(x)|Dv|^2.
\end{equation*}
Since the domain $\Omega$ is smooth, equivalently, $H^{1,2}_\rho$ is composed by the functions
$v$ which have distributional derivative with finite norm.
The space $H^{1,2}_{0,\rho}$ is defined as the completion of $C^\infty_0(\Omega)$ with respect to the norm
$\|\cdot \|_{H^{1,2}_\rho}$ and it is a reflexive Hilbert space.

\noindent
Moreover let $\mu \in C(\overline \Omega)$ be such that $0<\mu_{-}\leq\mu_{+}\leq1$
and let us define the function
\begin{equation}
	\label{eq:V}
	V_\mu[g,U](x):=\int_{U} \frac{g(y)}{|x-y|^{N(1-\mu(x))}}dy.
\end{equation}
By \cite[Theorem 3.1]{sam2} it follows that,  for any $1\leq q(x) \leq \infty,$  with
$$
\displaystyle \frac1{m(x)} - \frac1{q(x)}\leq\mu(x)
$$
it follows
\begin{equation}\label{eq:potential}
\|V_\mu[g,{\Omega}](x)\|_{q(\cdot)}\leq\Theta\|g\|_{m(\cdot)},
\end{equation}
for some positive constant $\Theta$ and for any $g\in L^{m(\cdot)}(\Omega)$. We can now prove the following

\begin{theorem}\label{thm: Sobolev}
Let $p(x)\geq 2$ for all $x\in\Omega$ and set
$$
\bar t:=\inf_{x\in\overline\Omega}\frac{p(x)-1}{p(x)-2}r,
$$
where $r>0$ is such that
\begin{equation}
	\label{eq:weight}
\int_{\Omega} \frac{1}{\rho^{t(x)}|x-y|^{\gamma}}\leq C(\gamma),\qquad 
\max_{x \in \overline{\Omega}} \frac{p(x)-2}{p(x)-1}  \leq r<1, 
\qquad t(x):=\frac{p(x)-1}{p(x)-2}r,
\end{equation}
with $N-2\bar t<\gamma< N-2$ if $N\geq 3$ and
$\gamma =0$  if $N=2$. Then,  for any $w\in H^{1,2}_{0,\rho}(\Omega)$, we have
\begin{equation}\label{FTFTnddjncj}
\|w\|_{q(\cdot)}\leq C\Big( \int_{\Omega}\rho|Dw|^2\Big)^{\frac 12},
\end{equation}
for some positive constant $C$ and any $1\leq q(\cdot)< 2^*(\bar t)$, where
\begin{equation}\label{eq:2*}
\frac{1}{2^*(\bar t)}=\frac 12 - \frac 1N + \frac 1{\bar t} \left( \frac 12 - \frac{\gamma}{2N}\right).
\end{equation}
Furthermore the embedding of $H^{1,2}_{0,\rho}(\Omega)$ into $L^{q(\cdot)}(\Omega)$ is compact.
\end{theorem}

\begin{proof}
We can assume   that $w\in C_c^1(\Omega)$.
Hence  standard potential estimates (see \cite[Lemma 7.14]{GT}) give
\begin{equation*}
	%\label{eq:hip_real}
	|w(x)|\leq C \int_{\Omega}\frac{|Dw(y)|}{|x-y|^{N-1}}dy,
\end{equation*}
where $C$ is a constant depending on the dimension $N$. Then
\begin{align*}
|w(x)|&\leq  C\int_{\Omega}\frac{|Dw(y)|}{|x-y|^{N-1}}dy \\
&\leq  C\int_{\Omega}\frac{1}{\rho^\frac 12|x-y|^{\frac{\gamma}{2\bar t}}}\frac{|Dw(y)|\rho^{\frac 12}}{|x-y|^{N-1-\frac{\gamma}{2\bar t}}}dy\\
&\leq  C\left(\int_{\Omega}\frac{1}{\rho^{\bar t}|x-y|^\gamma}dy\right)^{\frac{1}{2\bar t}}
\Big(\int_{\Omega}\frac{\big(|Dw(y)|\rho^{\frac 12}\big)^{(2\bar t)'}}{|x-y|^{(N-1-\frac{\gamma}{2\bar t})(2\bar t)'}}dy\Big)^{\frac{1}{(2\bar t)'}},
\end{align*} 
where  in the last inequality we used H\"older inequality with $\frac{1}{2\bar t}+\frac{1}{(2\bar t)'}=1$. Note that, by the definition of~$\bar t$ and by \eqref{eq:weight}, it follows that
$\int_{\Omega} \frac{1}{\rho^{\bar t}|x-y|^{\gamma}}\leq C.$
Hence
\begin{equation}
	\label{eq:w1}
|w(x)|\leq  C \Big(\int_{\Omega}
\frac{\big(|Dw(y)|\rho^{\frac 12}\big)^{(2\bar t)'}}{|x-y|^{(N-1-\frac{\gamma}{2\bar t})(2\bar t)'}}dy\Big)^{\frac{1}{(2\bar t)'}}.
\end{equation}
We point out that
\begin{equation}\label{eq:f_summability}
(|Dw|\rho^{\frac 12})^{(2\bar t)'}\in L^{\frac{2}{(2\bar t)'}}(\Omega).
\end{equation}
From \eqref{eq:w1}, by using equation \eqref{eq:V} with $\mu=1 -\frac{1}{N}(N-1-\frac{\gamma}{2\bar t})(2\bar t)'$, we obtain
\begin{equation*}
%	\label{eq:w2}
|w(x)|\leq  C\big(V_\mu\big[\big(|D w(y)|\rho^\frac12\big)^{(2\bar t)'},\Omega\big](x)\big)^\frac{1}{(2\bar t)'}\,.
\end{equation*}
Since $\gamma>N-2\bar t$,
we also have $N\bar t-2N+2\bar t+\gamma>0$ and $\mu>0$.
We shall use now \eqref{eq:potential} (see  \cite[Theorem 3.1]{sam2}) with
$\frac 1{m}=(2\bar t)'/2$, see~\eqref{eq:f_summability}.
Let us now fix an arbitrary $\tilde q(\cdot)>1$ such that
$1/m -1/{\tilde q(\cdot)} \leq\mu$, which is possible since  $1/m-\mu<1$, 
as follows by $N\bar t-2N+2\bar t+\gamma>0$. Therefore, we have
\begin{align}
	\label{eq:ws1}
\|w(x)\|_{\tilde q(\cdot)(2\bar t)'}&\leq
C\Big\|\Big(V_\mu\Big[\Big(|D w(y)|\rho^\frac12\Big)^{(2\bar t)'},\Omega\Big](x)\Big)^\frac{1}{(2\bar t)'}\Big\|_{\tilde q(\cdot)(2\bar t)'}\\\nonumber
& \leq C\Big\|V_\mu\Big[\Big(|D w(y)|\rho^\frac12\Big)^{(2\bar t)'},\Omega\Big](x)\Big\|_{\tilde q(\cdot)}^{\frac1{(2 \bar t)'}}
.\end{align}
From \eqref{eq:ws1}, by  \eqref{eq:potential}  we get
\begin{equation*}
	%\label{eq:wsshafjF1}
\|w\|_{\tilde q(\cdot)(2\bar t)'}\leq C\Big( \int_{\Omega}\rho|Dw|^2\Big)^{\frac 12},
\end{equation*}
that gives \eqref{FTFTnddjncj} and \eqref{eq:2*} with $q(x)= \tilde q(x)(2\bar t)'$, and consequently for any $q(\cdot)$ as in the statement of the theorem.
\noindent
Finally the compactness of the embedding follows arguing exactly as in \cite{CES1}.
\end{proof}

\subsection{The eigenvalue problem}\label{hfjhdbvjdhbvjhdf}
Let us consider the linearized operator
\begin{align*}
	%\label{lin}
\hspace{1cm} L_u(v,\varphi)&:= \int_\Omega |Du|^{p(x)-2}(Dv, D\varphi)\\
&+\int_\Omega (p(x)-2) |Du|^{p(x)-4}(Du,Dv)(Du,D\varphi) - \int_\Omega \partial_sf(|x|,u)v \varphi, \nonumber
\end{align*}
for any $v,\varphi\in H^{1,2}_{0,\rho}$. We also define $\|\cdot\|_{A_u}$ to be the norm arising from the scalar product
\[
<v,\varphi >:=\int_\Omega |Du|^{p(x)-2}(Dv,D\varphi)+\int_\Omega (p(x)-2) |Du|^{p(x)-4}(Du,Dv)(Du,D\varphi),
\]
that is a norm equivalent to $\|v\|_{H^{1,2}_{0,\rho}}=\big( \int_{\Omega}\rho|Dv|^2\big)^{\frac 12}$.
\noindent
Since $\partial_sf(|x|,u) \in L^\infty (\Omega)$,
the first eigenvalue $\mu_1(u)$ of the linearized operator
is well defined by
$$
\mu_1(u)=\underset{\phi \in H^{1,2}_{0,\rho} \backslash \{0\}}{\inf} R_u(\phi),\qquad
R_u(\phi)=\frac{\|\phi\|_{A_u}^2 -\int_\Omega  \partial_sf(|x|,u) \phi^2}{\int_{\Omega}
\phi^2}.$$
Consider now a minimizing sequence $\phi_n \in H^{1,2}_{0,\rho}$, $\int_\Omega \phi_n^2=1$, with
$R_u(\phi_n)$ converging to $\mu_1(u)$  as $n\to\infty.$
Since $\partial_sf(|x|,u) \in L^\infty (\Omega)$, we have that the sequence
$(\|\phi_n\|_{A_u})$ remains bounded. Therefore, up to a subsequence, we get that
$\phi_n \rightharpoonup \phi_1$ weakly in $H^{1,2}_{0,\rho}$
and therefore $
\phi_n \rightarrow \phi_1$ strongly in  $L^2(\Omega)$ (by combining 
the assertions of Lemma~\ref{le:1/gradient} and Theorem~\ref{thm: Sobolev}).
Now, the term $\int_\Omega  \partial_sf(|x|,u) \phi^2$ is continuous in
$L^2(\Omega)$ and $\|\cdot \|_{A_u}$ is weakly lower
semi-continuous in $H^{1,2}_{0,\rho}$. Therefore, $\phi_1 \in  H^{1,2}_{0,\rho}$ is such
that $\int_\Omega \phi_1^2=1$ and $R_u(\phi_1)\leq \mu_1(u)$. Hence, $\mu_1(u)$ is attained at $\phi_1$.
It is now standard to show that $\phi_1$ solves
$L_u(\phi_1,\varphi) =\int_{\Omega} \mu_1(u) \phi_1\varphi$ for any $\varphi\in H^{1,2}_{0,\rho}.$
Arguing now exactly as in \cite[see p.299]{CES1}, we get that every minimizer is of fixed sign and
the first eigenspace is one-dimensional.

\begin{remark}\rm
	\label{spectral}
Following \cite{CES2} it is now possible to develop a complete spectral theory for the linearized operator, showing that
it has an increasing  discrete sequence of eigenvalues with finite dimensional eigenspaces.
\end{remark}

\subsection{Proof of Theorem~\ref{radsymm} completed}

\noindent Let us write the solution $u=u(r,\theta)$ in polar coordinates,
where $r=|x|$ and $\theta=(\theta_1,\dots,\theta_{n-1})$ are the
$n-1$ angular variables.
\noindent Assume first that $u$ is Semi-stable according to Definition \ref{jhfgjdbjf}.
\noindent If $u$ was not radial,
then $u_{\theta_i} \not=0$ and $u_{\theta_i}$ changes sign, for some $i \in \{1,\dots,n-1 \}$.
Notice now that, since we are considering $C^1$ solutions, it is clear from the proof that \cite[Lemma 3.1]{CLS}
can be stated with $\beta\equiv 0$ and, passing to the limit, with $\varepsilon=0$ as well. In particular, we get
\begin{equation*}
%	\label{boh}
\int_\Omega |Du|^{p(x)-2} |Du_{\theta_i}|^2 \leq C \int_\Omega |Du|^{p(x)-2} \|D^2 u\|^2\leq C,
\end{equation*}
and by the boundary conditions we obtain $u_{\theta_i}\in H^{1,2}_{0,\rho}$.
It is now easy to see that, since $p(x)$ is radially 
symmetric, it follows, differentiating the equation in \eqref{P} 
with respect to $\theta_i$, that
\begin{equation}
	\label{uteta}
L_u (u_{\theta_i}, \varphi) = 0, \qquad \forall  \varphi \in H^{1,2}_{0,\rho}.
\end{equation}
In particular, $u_{\theta_i}$ is an eigenfunction of the
linearized operator corresponding to the $0$ eigenvalue. By the semi-stability assumption on $u$, this implies that
$u_{\theta_i}$ is the  first eigenfunction of $L_u$ and consequently (see Section \ref{hfjhdbvjdhbvjhdf}) it should have
constant sign in $\Omega$. This contradiction shows that $u$ is radially symmetric.
\noindent If else we assume that $u$ is non-degenerate, the conclusion follows in the
same way, noticing that $0$ is not an eigenvalue and therefore \eqref{uteta}
implies that $u_{\theta_i}=0$. \qed

\bigskip
\noindent
{\bf Acknowledgments.} The authors wish to thank Prof.\ Petteri Harjulehto
and Dr.\ Michela Eleuteri for providing the useful bibliographic reference~\cite{sam2}.

\bigskip


\medskip

\begin{thebibliography}{99}
\smallskip

\bibitem{bwwi}
{\sc T.\ Bartsch, T.\ Weth, M.\ Willem},
Partial symmetry of least energy nodal solutions to some variational problems,
{\em J. Anal. Math.} {\bf 96} (2005), 1--18.

\bibitem{CLS}
{\sc A. Canino, P. Le, B. Sciunzi},
Local $W_{\rm loc}^{2m}$ regularity for $p(x)$-Laplace equations,
{\em Manuscripta Math.}, in press.

\bibitem{CES1}
{\sc D. Castorina, P. Esposito and B. Sciunzi},
Degenerate elliptic equations with singular nonlinearities,
{\em Calc. Var. Partial Differential Equations} {\bf 34} (2009), 279--306.

\bibitem{CES2}
{\sc D. Castorina, P. Esposito and B. Sciunzi},
Spectral theory for linearized p-Laplace equations,
{\em Nonlinear Anal. TMA} {\bf 74} (2011), 3606--3613.

\bibitem{chabr}
{\sc J. Chabrowski, Y. Fu},
Existence of solutions for $p(x)$-Laplacian problems on a bounded domain,
{\em J. Math. Anal. Appl.} {\bf 306} (2005), 604--618.

\bibitem{DS}
{\sc L. Damascelli, B. Sciunzi},
Regularity, monotonicity and symmetry of positive solutions
of $m$-Laplace equations,
{\em J. Differential Equations} {\bf 206} (2004), 483--515.

\bibitem{book}
{\sc L.\ Diening, P.\ Harjulehto, P.\ H\"ast\"o, M. Ru\v zi\v cka},
Lebesgue and Sobolev spaces with variable exponents, Lecture Notes in Mathematics,
{\bf 2017}, Springer-Verlag, Heidelberg, 2011.

\bibitem{fan}
{\sc X. Fan},
Existence and uniqueness fot the $p(x)$-Lpalacian-Dirichlet problems,
{\em Math. Nachr.} {\bf 284} (2011), 1435--1445.

\bibitem{dino-sns}
{\sc A.\ Farina, E.\ Valdinoci, B.\ Sciunzi},
Bernstein and De Giorgi type problems: new results via a geometric approach,
{\em Annali Scuola Normale Superiore Pisa, Cl.Sci} {\bf 7} (2008)

\bibitem{FMS2}
{\sc A.\ Farina, B.\ Montoro, B.\ Sciunzi},
Monotonicity of solutions of quasilinear
degenerate elliptic equation in half-spaces,
{\em preprint}.

\bibitem{GNN}
{\sc B.~Gidas, W.~M.~Ni, and L.~Nirenberg},
Symmetry and related properties via the maximum principle,
{\em Comm.\ Math.\ Phys.\ } {\bf 68} (1979), 209--243.

\bibitem{GT}
{\sc D.~Gilbarg, N.~S.~Trudinger},
Elliptic Partial Differential Equations of Second Order.
{\em Reprint of the 1998 Edition}, Springer.

\bibitem{MSS}
{\sc L.\ Montoro, B.\ Sciunzi, M.\ Squassina},
Asymptotic symmetry for a class of quasi-linear parabolic problems,
{\em Advanced Nonlinear Studies} {\bf 10} (2010), 789--818.

\bibitem{sam2}{ \sc N.G. Samko,  S.G. Samko,  B.G. Vakulov},
Weighted {S}obolev theorem in {L}ebesgue spaces with variable exponent,
{\em J. Math. Anal. Appl.} {\bf 335} (2007), 560--583.

\bibitem{serrin}
{\sc J.~Serrin},
Weighted {S}obolev theorem in {L}ebesgue spaces with variable exponent,
{\em Arch. Rational Mech. Anal} {\bf 43} (1971), 304--318.

\bibitem{smetswillem}
{\sc D. Smets, M.\ Willem},
Partial symmetry and asymptotic behavior for some elliptic variational problems,
{\em Calc. Var. Partial Differential Equations} {\bf 18} (2003), 57--75.

\bibitem{marco-jean}
{\sc M. Squassina, J. Van Schaftingen},
Finding critical points whose polarization is a critical point,
{\em Topol. Meth. Nonlinear Anal.}, to appear.

\bibitem{jvs}
{\sc J. Van Schaftingen},
Symmetrization and minimax principles,
{\em Commun. Contemp. Math.} {\bf 7} (2005), 463--481.

\bibitem{jvs-topol}
{\sc J. Van Schaftingen},
Approximation of symmetrizations and symmetry of critical points,
{\em Topol. Methods Nonlinear Anal.} {\bf 28} (2006), 61--85.

\bibitem{myself}
{\sc M. Squassina},
Radial symmetry of minimax critical points for nonsmooth functionals,
{\em Commun. Contemp. Math.} {\bf 13} (2011), 487--508.

\bibitem{Zha}{\sc Q. Zhang},
 A strong maximum principle for differential equations with nonstandard {$p(x)$}-growth conditions,
{\em J. Math. Anal. Appl.} {\bf 312} (2005),  24--32.

\end{thebibliography}
\end{document}